\documentclass[11pt]{amsproc}

\usepackage{a4wide}
\usepackage[pagebackref=true,colorlinks,linkcolor=red,citecolor=blue,urlcolor=blue,hypertexnames=true]{hyperref}
\usepackage{setspace}
\usepackage{amsmath}
\usepackage{amssymb}
\usepackage{amsfonts}
\usepackage{amsthm}
\usepackage{mathtools}
\usepackage{array}

\theoremstyle{plain}
\newtheorem{thm}{Theorem}[section]
\theoremstyle{definition}

\newtheorem{lem}[thm]{Lemma}
\newtheorem{prop}[thm]{Proposition}
\newtheorem{cor}[thm]{Corollary}
\theoremstyle{remark}
\newtheorem{rem}[thm]{Remark}

\newtheorem*{quest}{Question}
\newtheorem*{obs}{Observation}

\newcommand{\norm}[1]{\left\Vert#1\right\Vert}
\newcommand{\set}[1]{\left\{#1\right\}}

\newcommand{\bbN}{\mathbb{N}}

\DeclareMathOperator{\Proj}{Proj}
\DeclareMathOperator{\supp}{supp}

\title[Strong uncountable cofinality]{Strong uncountable cofinality for unitary groups\\ of von Neumann algebras}

\author{Philip A. Dowerk}
\address{P.A.D., Institute of Geometry, TU Dresden, 01062 Dresden, Germany}
\email{philip.dowerk@tu-dresden.de}

\begin{document}

\onehalfspace

\begin{abstract}
 We show that unitary groups of II$_1$ factors and of properly infinite von Neumann algebras have strong uncountable cofinality. In particular, we obtain a short alternative proof for the strong uncountable cofinality of $U(\ell^2(\mathbb{N}))$, which was first proven by Ricard and Rosendal.
\end{abstract}

\maketitle


\section{Introduction}\label{sec_intro}

In this article we study the algebraic structure of unitary groups of von Neumann algebras. 
Following the work of Droste and G\"obel \cite{DG-05} we say that a group $G$ has \textit{strong uncountable cofinality} if for every countable exhaustive increasing chain $W_1\subseteq W_2\subseteq\ldots\subseteq G=\bigcup_{n\in\mathbb{N}}W_n$ there exist $k,n\in\mathbb{N}$ such that $G=W_n^k\coloneqq\set{w_1\cdots w_k\mid w_i\in W_n\mbox{ for }1\leq i\leq k}$. If $k$ can be chosen independently of the sequence $(W_n)_{n\in\mathbb{N}}$ we say that $G$ has \textit{$k$-strong uncountable cofinality}. 
As noticed by Droste and Holland \cite{DH-05} the strong uncountable cofinality of a group $G$ is equivalent to having both \textit{uncountable cofinality} (i.e. cannot be written as a countable increasing union of proper subgroups) and having the \textit{Bergman property} (i.e. for every symmetric unital generating set $S$ of $G$ there exists $k\in\mathbb{N}$ such that any element in the group can be expressed as a word of length $k$ in $S$).

The Bergman property has been first observed by Bergman \cite{B-06} for $S_\infty$.
Since then it has received a lot of attention, see, for example, \cite{BTV-12, B-06, dC-06, DG-05, DH-05, DHU-08, DT-09, KR-07, M-04, RR-07, R-09, To-06, T-06}. 
The notion of strong uncountable cofinality has several equivalent geometric reformulations, for example, whenever the group acts by isometries on a metric space, every orbit is bounded (due to de Cornulier \cite{dC-06} and Pestov, see also \cite[Theorem 1.2]{R-09}). 
Further, the strong uncountable cofinality implies the well-known fixed point properties (FA) and (FH) as observed by Rosendal \cite[Proposition 3.3]{R-09}. \\

The first result concerning strong uncountable cofinality for unitary groups of von Neumann algebras was given by Ricard and Rosendal \cite[Theorem 1]{RR-07}, where they showed that $\mathrm{U}(\ell^2(\mathbb{N}))$ has 200-strong uncountable cofinality. In this article we complete the picture for infinite-dimensional von Neumann algebras. 

The finite-dimensional unitary groups $\mathrm{U}(n),\ n\in\mathbb{N},\ n\geq 2,$ fail to have strong uncountable cofinality for two reasons: As pointed out by Yves de Cornulier in a private communication, one can use arguments from the proof of \cite[Theorem 1.7]{SST-97} by Saxl, Shelah and Thomas to show that any uncountable linear group has countable cofinality. 
Additionally nontrivial compact connected groups do not have the Bergman property by the work of Schneider \cite{S-19}. 
In contrast, profinite groups (i.e. compact totally disconnected groups) can have uncountable cofinality, see the work of Koppelberg and Tits \cite{KT-74}. \\

Let us state the main results of this article. 
We do not require any separability assumptions in the theorems below.

\begin{thm}\label{thm_main}
 The unitary group of a type $\mathrm{II}_1$ factor has strong uncountable cofinality. 
\end{thm}

However, as observed in Remark \ref{rem_II_1}, the unitary group of a type II$_1$ factor does not have $k$-strong uncountable cofinality for any $k\in\mathbb{N}$. 

\begin{thm}\label{thm_main_inf}
 The unitary group of a properly infinite von Neumann algebra has 96-strong uncountable cofinality.
\end{thm}

In particular, Theorem \ref{thm_main_inf} includes the result \cite[Theorem 1]{RR-07} of Ricard and Rosendal (which is the separable type I$_\infty$ factor case, i.e. $\mathrm{U}(\ell^2(\mathbb{N}))$) with a better constant, and we give a short alternative proof of that fact. 
By \cite[Proposition 3.3]{R-09} we have that Theorems \ref{thm_main} and \ref{thm_main_inf} imply properties (FA), (FH), and algebraic property (FH) (i.e. no continuity assumption on the action by isometries on a real Hilbert space) for  unitary groups of type II$_1$ factors and of properly infinite von Neumann algebras.

Moreover, Theorems \ref{thm_main} and \ref{thm_main_inf} together with the equivalent reformulations of strong uncountable cofinality summarized in \cite[Theorem 1.2]{R-09} imply the following result.

\begin{cor}\label{cor_main}
 Let $G$ denote the unitary group of a type II$_1$ factor or of a properly infinite von Neumann algebra. Then the following hold.
\begin{itemize}
 \item[(i)] Every left-invariant metric on $G$ is bounded.
 \item[(ii)] Whenever $G$ acts by isometries on a metric space, every orbit is bounded.
 \item[(iii)] Whenever $G$ acts on a metric space by mappings that are Lipschitz for large distances,  every orbit is bounded.
 \item[(iv)] Whenever $G$ acts by uniform homeomorphisms on a geodesic space, every orbit is bounded. 
\end{itemize}
\end{cor}

The fact that strong uncountable cofinality passes to quotients by normal subgroups implies that Theorem \ref{thm_main}, Theorem \ref{thm_main_inf} and thus Corollary \ref{cor_main} also hold for projective unitary groups of infinite-dimensional factors (i.e. the quotient of $\mathrm{U}(M)$ by its center), and for the unitary group of the Calkin algebra (as it is the quotient of $\mathrm{U}(\mathcal{H})$, $\mathcal{H}$ an infinite-dimensional Hilbert space, by the normal subgroup of compact perturbations of the identity).\\

The article is structured as follows. In Section \ref{sec_BNG} we obtain a new criterion for the topological Bergman property for Polish groups, in Section \ref{sec_Bergman} we establish a general criterion for strong uncountable cofinality, which we will apply in Section \ref{sec_factors} to obtain Theorem \ref{thm_main} and Theorem \ref{thm_main_inf}. We conclude our article in Section \ref{sec_open} with open questions on separable group topologies for unitary groups of von Neumann algebras. Preliminaries are presented as needed at the beginning of Sections \ref{sec_BNG}, \ref{sec_Bergman} and \ref{sec_factors}.

\section{The topological Bergman property for groups\\ with bounded normal generation}\label{sec_BNG}
The aim of this section is to provide a criterion for a Polish SIN group to have the topological Bergman property (a weaker version of strong uncountable cofinality), using a short and direct argument. This section is independent from the other sections, but the notion of bounded normal generation defined below appears in remarks in subsequent sections. 

A topological group is called \textit{Polish} if its underlying topology is separable and completely metrizable. We say that a topological group $G$ is \textit {SIN} (short for Small Invariant Neighborhoods) if $G$ admits a neighborhood basis at the identity element of $G$ consisting of conjugacy-invariant open neighborhoods. For basics on topological groups we refer the reader to \cite{AT08}.

Recall from \cite{R-09} that a topological group $G$ has the \textit{topological Bergman property} if for every increasing exhaustive sequence of symmetric subsets $W_0\subseteq W_1\subseteq\ldots\subseteq G$ with the Baire property, there are $n,k\in\mathbb{N}$ such that $G=W_n^k$. If $k$ can be chosen independently of the sequence $(W_n)_{n\in\mathbb{N}}$, then we say that $G$ has the $k$-\textit{topological Bergman property}. 
By \cite[Theorem 1.4]{R-09} the following properties are equivalent:
\begin{itemize}
 \item[(1)] $G$ has the topological Bergman property.
 \item[(2)] Whenever $G$ acts by isometries on a metric space $(X,d)$, such that for every $x\in X$ the function $G\rightarrow X,\ g\mapsto gx$, is continuous, then every orbit is bounded.
 \item[(3)] Any compatible left-invariant metric on $G$ is bounded.
 \item[(4)] $G$ is finitely generated of bounded width over any non-empty open subset.
\end{itemize}

The topological Bergman property implies property (FH), see \cite[Proposition 3.3]{R-09}.

As observed in \cite{R-09}, for a locally compact group $G$ the topological Bergman property is equivalent to compactness of $G$ (but compact connected groups never possess strong uncountable cofinality as remarked in Section \ref{sec_intro}). Some non-locally compact examples of groups with the topological Bergman property are the isometry group of the Urysohn space of diameter 1, endowed with the topology of pointwise convergence, see \cite[Theorem 1.5]{R-09}, and the homeomorphism group of the Hilbert cube, endowed with the topology of uniform convergence, see \cite[Theorem 1.6]{R-09}. 

Following \cite{DT-15}, we say that a group $G$ has \textit{property (BNG)} (short for \textit{bounded normal generation}) if for every element $g\in G\setminus\{1\}$ there exists $k\in\mathbb{N}$ such that $G=(g^G\cup g^{-G})^k$, where $g^G:=\{hgh^{-1}\mid h\in G\}$ denotes the conjugacy class of $g$ in $G$ and $g^{-G}:=(g^{-1})^G$, and $1$ is the neutral element of $G$. 
If the exponent $k$ is independent of the choice of $g\in G$, we say that $G$ has \textit{property $k$-(BNG)}.
It is easy to see that property (BNG) implies simplicity of the group, but the converse is not true, as observed in \cite{DT-15} (take, for example, the alternating group of finitely supported permutations on $\mathbb{N}$). 

\begin{prop}\label{prop_Bergman}
 Every Polish SIN group $G$ with property (BNG) has the topological Bergman property. If $G$ is $k$-(BNG) for some fixed $k\in\mathbb{N}$, then it is $2k$- topologically Bergman. 
\end{prop}
\begin{proof}
 Let $B_1\subseteq B_2\subseteq \ldots \subseteq G$ be an exhaustive sequence of symmetric subsets with the Baire property. By the Baire category theorem there exists $n\in\bbN$ such that ${\rm int}(B_n)\neq\emptyset$. By Pettis' Theorem \cite[Theorem 9.9]{Kec} $B_n^{-1}B_n=B_n^2$ contains an open neighborhood of $1$. Since $G$ is SIN, ${\rm int}(B_n^2)$ contains a conjugacy-invariant symmetric open neighborhood $U$ of the identity. Property (BNG) now implies that there exists $k\in\bbN$ such that $U^k=G$, thus in particular $B_n^{2k}=G$ (here $k$ is independent of $U$ and of the sequence $(B_n)_{n\in\mathbb{N}}$ if $G$ is $k$-(BNG)). 
\end{proof}

Examples of groups that are topologically Bergman thus include projective unitary groups of separable II$_1$ factors with the strong operator topology by \cite[Theorem 1.2]{DT-15}. However, these groups are not $k$-topologically Bergman for any $k\in\mathbb{N}$ (as one can take open balls of arbitrarily small positive radius to generate the group).

\begin{rem}
 As defined in \cite{D15}, we say that a topological group $G$ has the \textit{topological bounded normal generation property} if for every $g\in G\setminus\{1\}$ there exists $k\in\mathbb{N}$ such that $G=\overline{(g^G\cup g^{-G})^k}$. A Polish group with the topological bounded normal generation property does not need to have the topological Bergman property: examples are derived $L^1$ full groups of hyperfinite ergodic graphings, see \cite[Theorem A]{DLM-17} (these admit a complete right-invariant unbounded metric inducing a Polish topology). 
\end{rem}

\section{A general criterion for strong uncountable cofinality}\label{sec_Bergman}
In this section we present a noncommutative analogue of \cite[Theorem 2.1]{DHU-08}. It is essential in our proof of Theorem \ref{thm_main} and Theorem \ref{thm_main_inf} in Section \ref{sec_factors}. Recall that for a group $G$ we denote by $g^G$ the conjugacy class of $g\in G$ and $g^{-G}\coloneqq (g^{-1})^G$, and we write $g^h\coloneqq hgh^{-1}$ for $g,h\in G$. \\ 

Let $\mathcal{H}$ be an infinite-dimensional Hilbert space (not necessarily separable). Let $\mathrm{B}(\mathcal{H})$ be the algebra of bounded linear operators on $\mathcal{H}$. We denote by $\mathrm{U}(\mathcal{H})$ (or $\mathrm{U}(\mathrm{B}(\mathcal{H}))$) its unitary group, i.e. the group with underlying set $\{g\in \mathrm{B}(\mathcal{H})\mid gg^*=g^*g=1\}$ and group operation given by composition of operators, where $1$ denotes the identity operator and $g^*$ denotes the adjoint operator of $g$. We write $\Proj(\mathrm{B}(\mathcal{H}))$ for the set of projections in $\mathrm{B}(\mathcal{H})$. If $p,q\in\Proj(\mathrm{B}(\mathcal{H}))$, then we will frequently write $p^{\bot}$ for $1-p$ and $p\leq q$ if $p$ is a subprojection of $q$. For two projections $p,q\in \Proj(\mathrm{B}(\mathcal{H}))$ we denote by $p\wedge q$ the projection onto $p\mathcal{H}\cap q\mathcal{H}$.  

Let $G$ denote a subgroup of $\mathrm{U}(\mathcal{H})$. Let $g\in G$ and set $p_g$ to be the projection onto the subspace $\{\xi\in\mathcal{H}\mid g\xi=\xi\}$ of $g$-fixed points in $\mathcal{H}$. 
For $g\in G$ we denote $\supp g\coloneqq 1-p_g$ and call it the \textit{support} of $g$ (not to be confused with the usual terminology of a support projection of an element in a von Neumann algebra, which is always 1 for a unitary). 
For a projection $p\in \Proj(\mathrm{B}(\mathcal{H}))$ we denote by $G(p)\coloneqq \{g\in G\mid \supp g\leq p\}$ the subgroup of elements of $G$ with support at most $p$. Thus we have $G(p)=G\cap\left(\mathrm{U}(p\mathrm{B}(\mathcal{H})p)\oplus(1-p)\right)$. 

\begin{thm}\label{thm_Bergmann}
 Let $G$ be a subgroup of the unitary group $\mathrm{U}(\mathcal{H})$ on an infinite-dimensional Hilbert space $\mathcal{H}$. Assume that there is an infinite set of projections $\mathcal{P}\subseteq \Proj(\mathrm{B}(\mathcal{H}))$ such that the following conditions are satisfied for $G$ and $\mathcal{P}$:
\begin{itemize}
 \item[(A)] There exists a sequence $(p_n)_{n\in\mathbb{N}}$ of mutually orthogonal nonzero projections in $\mathcal{P}$ summing to the identity such that for each sequence $(g_n)_{n\in\mathbb{N}}\subseteq G$ with $g_np_n=p_ng_n$ one has $\oplus_{n\in\mathbb{N}}g_{n}p_n \in G$. 
 \item[(B)] For each nonzero $p\in \mathcal{P}$ there are $n\in\mathbb{N}$, some nonzero subprojection $q\leq p,\ q\in\mathcal{P},$ and $g\in G(p)$ such that $G(q)\subseteq (g^{G(p)}\cup g^{-G(p)})^n$. 
 \item[(C)] For each nonzero projection $p\in P$ there are $m\in\mathbb{N}$ and some finite set $F\subseteq G$ such that $G=(G(p)\cup F)^m$. 
\end{itemize}
 Then $G$ has strong uncountable cofinality. 
 Moreover, if there are fixed exponents $n$ in (B) and $m$ in (C), then $G$ has $3nm$-strong uncountable cofinality. 
\end{thm}

Note that condition (C) above implies that the subgroups $G(p)$ are nontrivial for every nonzero $p\in\mathcal{P}$. The Hilbert space $\mathcal{H}$ in Theorem \ref{thm_Bergmann} does not have to be separable.  

\begin{rem}
 As a special case of Theorem \ref{thm_Bergmann} one can take the group $G$ to be the group $\mathrm{S}_\infty$ of all bijections $\mathbb{N}\rightarrow\mathbb{N}$ as in \cite{B-06}, which embeds into $\mathrm{U}(\ell^2(\mathbb{N}))$ by permuting the base vectors. 
 
  Another example is the full group $G$ of measure preserving ergodic transformations as in \cite{DHU-08}, which is a subgroup of $\mathrm{Aut}(X,\mu)$ for some standard measure space $(X,\mu)$ (so it embeds naturally into $\mathrm{U}(\mathrm{L}^2(X,\mu))$ by $u_g(\xi)=\xi\circ g^{-1}$ for $\xi\in \mathrm{L}^2(X,\mu)$ for $g\in G$). 
\end{rem}

Before being able to prove Theorem \ref{thm_Bergmann} we need to prove two lemmas.\\ 

Let $(W_n)_{n\in\mathbb{N}}$ be a countable increasing chain of symmetric subsets such that $G=\bigcup_{n\in\mathbb{N}}W_n$, where $G$ is a subgroup of the unitary group $\mathrm{U}(\mathcal{H})$ on an infinite-dimensional Hilbert space $\mathcal{H}$. 
We say that a projection $p\in \Proj(\mathrm{B}(\mathcal{H}))$ is \textit{full for} a subset $W\subseteq G$ if for every $g\in G(p)$ there is some $h\in W$ such that $gp=hp$. 

\begin{lem}\label{lem_1}
 Assume that $G$, $\mathcal{P}$ and $(p_n)_{n\in\mathbb{N}}$ satisfy condition (A) in Theorem \ref{thm_Bergmann}. 
  Then there exists $k\in\mathbb{N}$ such that $p_k$ is full for $W_k$. 
\end{lem}
\begin{proof}
 Assume to the contrary that no $p_n$ is full for $W_n$. Then for each $n$ we can choose $g_n\in G(p_n)$ such that $g_np_n\neq hp_n$ for any element $h\in W_n$. Condition (A) in Theorem \ref{thm_Bergmann} then implies $g:=\oplus_{n\in\mathbb{N}}g_{n}p_n\in G$. Since $(W_n)_{n\in\mathbb{N}}$ is exhaustive, there must be some $m\in\mathbb{N}$ such that $g\in W_m$. But $gp_m=g_m p_m$, contradicting our choice of the element $g_m$. Thus there is $k\in\mathbb{N}$ such that $p_k$ is full for $W_k$. 
\end{proof}

\begin{lem}\label{lem_2}
 Assume that $G$, $\mathcal{P}$ and $(p_n)_{n\in\mathbb{N}}$ satisfy conditions (A) and (B) in Theorem \ref{thm_Bergmann}. Let $p_k$ be as in Lemma \ref{lem_1}. Then for every subprojection $p\leq p_k,\ p\in \mathcal{P}$, there exists $l\in\mathbb{N}$ such that $G(p)\subseteq W_l^{3n}$. 
\end{lem}
\begin{proof}
 Let $k\in\mathbb{N}$ and $p_k\in \mathcal{P}$ as in Lemma \ref{lem_1}. Apply now (B) to obtain $n\in\mathbb{N}$, a subprojection $p\leq p_k$ and some $g\in G(p_k)$ such that every element $h$ of $G(p)$ is a product of $n$ conjugates of $g$ or $g^{-1}$ with some conjugating elements $g_1,\ldots,g_n$ from $G(p_k)$. 
 Consider now such an element $h=g^{\pm g_1}\cdots g^{\pm g_n}\in G(p)$. Let $m$ be such that $g\in W_m$ and set $l\coloneqq\max\{k,m\}$. 
 Since $p_k$ is full for $W_k$ there are elements $w_i\in W_k$, $i=1,\ldots,n$, such that $w_{i}p_k=g_{i}p_k$ for $i=1,\ldots,n$. Since $g,h\in G(p_k)$, we actually have $h=g^{\pm w_1}\cdots g^{\pm w_n}\in W_l^{3n}$ and thus $G(p)\subseteq W_l^{3n}$. 
\end{proof}

We are now ready to prove Theorem \ref{thm_Bergmann}.

\begin{proof}[Proof of Theorem \ref{thm_Bergmann}:]
 Use Lemma \ref{lem_2} to find some projection $p\in \mathcal{P}$ and $l\in\mathbb{N}$ such that $G(p)\subseteq W_l^{3n}$.
 Let $m$ and $F\subseteq G$ be as in (C). Finiteness of $F$ implies that there is some $j\in\mathbb{N}$, $j\geq l$, such that $G(p)\cup F\subseteq W_j^{3n}$. Hence $G=(G(p)\cup F)^m\subseteq W_j^{3nm}$, i.e. $G$ has strong uncountable cofinality. In particular, if $m$ and $n$ are fixed (thus independent of $p$), then $G$ has $3nm$-strong uncountable cofinality. 
\end{proof}

\section{The proofs of Theorem \ref{thm_main} and Theorem \ref{thm_main_inf}}\label{sec_factors}
We prove Theorem \ref{thm_main} in Subsection \ref{sec_II_1} and Theorem \ref{thm_main_inf} in Subsection \ref{sec_inf}.\\

We recall a few basics of the theory of von Neumann algebras, for a complete picture we refer to the reader to \cite{Tak}. 
A \textit{von Neumann algebra} $M$ on a Hilbert space $\mathcal{H}$ is a unital $*$-subalgebra of the algebra $\mathrm{B}(\mathcal{H})$ of bounded linear operators on $\mathcal{H}$ such that $M=M''$, where $M'\coloneqq \{x\in\mathrm{B}(\mathcal{H})\mid xy=yx\mbox{ for all }y\in M\}$ denotes the \textit{commutant} of $M$ and $M''\coloneqq (M')'$.
Let us say that two projections $p,q\in M$ are \textit{equivalent} (in the sense of Murray and von Neumann), written $p\sim q$, if there exists a partial isometry $x\in M$ such that $x^*x=p$ and $xx^*=q$. 
We say that a von Neumann algebra $M$ is \textit{finite} if every isometry is a unitary, i.e. the identity $1$ is not equivalent to a proper subprojection in $M$. Else we say that $M$ is \textit{infinite}. A von Neumann algebra $M$ is called \textit{properly infinite} if every central projection in $M$ is equivalent to some proper subprojection of itself. 
For example, the matrix algebra $M_n(\mathbb{C})$, $n\in\mathbb{N}$, is finite as a von Neumann algebra and for an infinite-dimensional Hilbert space $\mathcal{H}$ the algebra $\mathrm{B}(\mathcal{H})$ is properly infinite as a von Neumann algebra (as $\mathrm{B}(\mathcal{H})$ contains the unilateral shift and there are no central projections other than $0$ and $1$).
 
A \textit{factor} is a von Neumann algebra with trivial center $\mathbb{C}\cdot 1$, where $1$ denotes the identity operator of $M$. Factors form the fundamental building blocks of von Neumann algebras (at least in the case that the underlying Hilbert space is separable). 
The algebra $M_n(\mathbb{C})$ is a finite factor for any $n\in\mathbb{N}$ and $\mathrm{B}(\mathcal{H})$ is a properly infinite factor for any infinite-dimensional Hilbert space $\mathcal{H}$. 
However, not every finite factor is of the form $M_n(\mathbb{C})$ for some $n\in\mathbb{N}$. 
A \textit{type II$_1$ factor} $M$ is an infinite-dimensional factor that admits a normalized trace $\tau:M\to\mathbb{C}$, i.e. $\tau$ is a positive linear functional with $\tau(1)=1$ and $\tau(xy)=\tau(yx)$ for all $x,y\in M$. A type II$_1$ factor $M$ is finite as a von Neumann algebra and has the special property that the image of the set of projections $\Proj(M)$ in $M$ under the trace $\tau$ is the whole interval $[0,1]$. This justifies to speak of continuous dimension when viewing $\tau$ as a dimension function on the projection lattice. Examples of type II$_1$ factors can be obtained as the \textit{group von Neumann algebra} $L(G)\coloneqq \{\lambda(g)\mid g\in G\}''$. Here $G$ is any \textit{ICC group}, i.e. a group such that every conjugacy class except the one of the neutral element is infinite, and $\lambda:G\to \mathrm{U}(\ell^2(G)),\ \lambda(g)\xi_h=\xi_{gh}$ for $g,h \in G,\  \xi_h\in\ell^2(G)$, denotes the left-regular representation. The trace on $L(G)$ is given by $\tau(x)=\langle x\xi_1,\xi_1\rangle$ for $x\in L(G)$. For example, taking $G$ to be the finitary symmetric group $S_{\mathrm{fin}}$ on $\mathbb{N}$, one obtains the so called \textit{hyperfinite II$_1$ factor}, which is the smallest infinite-dimensional factor by \cite[Corollary 2]{C-76}. The celebrated result \cite[Corollary 7.2]{C-76} by Connes states that $L(G)$ is isomorphic to the hyperfinite II$_1$ factor if and only if $G$ is amenable and ICC.
While a type II$_1$ factor never contains a copy of $\mathrm{B}(\mathcal{K})$ for some infinite-dimensional Hilbert space $\mathcal{K}$, a properly infinite factor always does. 
It has been a difficult problem at the beginning of the theory to find any examples of properly infinite factors other than $\mathrm{B}(\mathcal{H})$ and we refer the reader to \cite{Tak} for such examples. The classes of finite-dimensional von Neumann algebras, type II$_1$ factors and properly infinite von Neumann algebras are mutually disjoint. \\

Let us explain how Theorem \ref{thm_Bergmann} can be used in the context of unitary groups of von Neumann algebras.
Let $M\subseteq \mathrm{B}(\mathcal{H})$ be a von Neumann algebra, where $\mathcal{H}$ is an infinite-dimensional Hilbert space. The unitary group $G\coloneqq\mathrm{U}(M)=\{g\in M\mid gg^*=g^*g=1\}$ of $M$ then acts naturally on $\mathcal{H}$.
For a projection $p\in M$ we have $G(p)= \mathrm{U}(pMp)\oplus(1-p)$. 

Note that condition (A) in Theorem \ref{thm_Bergmann} is automatically satisfied for projections in an infinite-dimensional von Neumann algebra. (If $M$ is finite-dimensional, then there is no countable sequence of mutually orthogonal nonzero projections.) 
Hence, to verify strong uncountable cofinality in the context of von Neumann algebras, it suffices to show that conditions (B) and (C) in Theorem \ref{thm_Bergmann} hold. 
If $M$ is of type II$_1$, then we let $\mathcal{P}$ be the set of all projections in $M$.
But if $M$ is a properly infinite von Neumann algebra, then condition (C) in Theorem \ref{thm_Bergmann} is never satisfied for a finite projection $p$ (i.e. $p\sim q\leq p$ implies $p=q$). In this case we need to restrict the set of projections $\mathcal{P}$ under consideration to be the set of projections $p\in M$ such that both $p$ and $1-p$ are infinite (i.e. not finite).

\subsection{Strong uncountable cofinality for type II$_1$ factors}\label{sec_II_1}
The aim here is to show that the unitary group $G$ of a II$_1$ factor $M$ has strong uncountable cofinality, using the results from the previous section. We denote the faithful normalized trace of $M$ by $\tau$. \\

The following lemma is a crucial step in our verification of condition (C) in Theorem \ref{thm_Bergmann}.

\begin{lem}\label{lem_masterful}
 Let $p\in M$ be a nonzero projection and decompose it orthogonally into $p=p_1+p_2+p_3$ with $\tau(p_1)=\tau(p_2)$, where $p_i$ are nonzero projections in $M$ for $i=1,2,3$. Then
 $$G(p)=(G(p_1+p_2)\cup G(p_2+p_3))^5.$$ 
\end{lem}

\begin{proof}
 Let $u\in G(p)$ be arbitrary and set $q_i\coloneqq up_iu^{-1}$ for $i=1,2,3$. Let $e\coloneqq p_2+p_3$.
 Since $\tau(q_1)=\tau(p_2)$ , we have  $\tau(e\wedge q_1^{\bot})\geq \tau(p_3)$ and hence
 $$\tau(e-e\wedge q_1^{\bot})\leq \tau(p_2).$$ 
 Thus there exists $u_1\in G(p_2+p_3)$ such that 
 $$u_1(e-e\wedge q_1^{\bot})u_1^{-1}=e-e\wedge u_1q_1^{\bot}u_1^{-1}\leq p_2.$$
 Observe that $u_1q_1u_1^{-1}\bot p_3$, since else we would have $u_1q_1u_1^{-1} p_3\neq 0,\ p_3 u_1q_1u_1^{-1}\neq 0$ and thus $u_1(e-e\wedge q_1^{\bot})u_1^{-1}$ would not be a subprojection of $p_2$, but only of $p_2+p_3$. \\ 
 Hence there exists $u_2\in G(p_1+p_2)$ such that $u_2u_1q_1u_1^{-1}u_2^{-1}=p_1$ and $u_2u_1(q_2+q_3)u_1^{-1}u_2^{-1}\leq p_2+p_3$. So we can find $u_3\in G(p_2+p_3)$ with $u_3u_2u_1q_2u_1^{-1}u_2^{-1}u_3^{-1}=p_2$ and thus with $v\coloneqq u_3u_2u_1$ we have 
 $$vq_iv^{-1}=p_i=u^{-1}q_iu \mbox{ for }i=1,2,3.$$ 
 Thus $vup_i=p_i vu$, i.e. there exist $v_1\in G(p_1),\ v_2\in G(p_2),\ v_3\in G(p_3)$ such that $vu=v_1v_2v_3$, i.e. $u=u_1^{-1}u_2^{-1}u_3^{-1}v_1v_2v_3$, with $u_2,v_1v_2\in G(p_1+p_2)$ and $u_1,u_3, v_3\in G(p_2+p_3)$. This proves our claim as $u\in (G(p_1+p_2)\cup G(p_2+p_3))^5$.  
\end{proof}

\begin{proof}[Proof of Theorem \ref{thm_main}:]
 Let $G$ denote the unitary group of a II$_1$ factor. 	
 Since condition (A) in Theorem \ref{thm_Bergmann} is always satisfied for unitary groups of type II$_1$ factors, it suffices to show that $G$ satisfies conditions (B) and (C) in Theorem \ref{thm_Bergmann}.  
 
 Condition (B) follows from \cite[Corollary 3.2]{DT-15} (which improves upon \cite[Theorem 1]{Bro-67}), using any symmetry of trace 0 in $G(p)$ - one needs 16 conjugacy classes of such a symmetry to generate any subgroup $G(q)\leq G(p)$ for $q\leq p$. 
 
 We now turn towards condition (C). Let $p\in \mathcal{P}$ be a nontrivial projection. Then there is a subprojection $q_0\leq p$ of trace $m=\tau(q_0)=(2/3)^n$ for some $n\geq 1$. Decompose $q_0=p_1+p_2$ orthogonally into subprojections of the same trace $m/2$. Since $1-m\geq m/2$ there is a subprojection $p_3\leq 1-q_0$ of trace $m/2$. 
 Now let $u\in G(p_1+p_2+p_3)$ be the unitary conjugating $p_1+p_2$ to $p_2+p_3$. 
 Then we have $G(p_2+p_3)=uG(p_1+p_2)u^{-1} \subseteq (G(q_0)\cup \{u,u^{-1}\})^3$. 
 Let $q_1\coloneqq p_1+p_2+p_3=q_0+p_3$ and use Lemma \ref{lem_masterful} to obtain
 $$G(q_1)=(G(q_0)\cup G(p_2+p_3))^5\subseteq (G(q_0)\cup \{u,u^{-1}\})^8.$$ 
 Observe that $\tau(q_1)=3m/2=(2/3)^{n-1}$. Repeating this construction we inductively obtain $q_0,\ldots,q_n$ with $\tau(q_n)=1$, i.e. $q_n=1$ and thus $G(q_n)=G$. Since by construction $G(q_{i+1})=(G(q_i)\cup\{u_i,u_i^{-1}\})^8$ for some $u_i\in G$, we obtain $G=(G(q_0)\cup F)^{8^n}$ for the finite set $F=\{ u_1,\ldots,u_n\}\subseteq G$. This implies that condition (C) holds, since $G(q_0)\subseteq G(p)$ and $p\in\mathcal{P}$ was arbitrary nonzero. 
\end{proof}

\begin{rem}\label{rem_II_1}
We observe that $G$ does not have $k$-strong uncountable cofinality for any fixed $k\in\mathbb{N}$, since the trace of $p$ in the above proof can be arbitrarily small and thus the exponent $n\in\mathbb{N}$ in the above proof can be arbitrarily large. 
\end{rem}

\begin{rem}
 If $G$ denotes the projective unitary group of a II$_1$ factor, then condition (B) in the above proof also follows from the bounded normal generation property, see \cite[Theorem 1.2]{DT-15}. This theorem implies that one can in fact choose any $g\in G(p)\setminus\{1\}$ to generate $G(p)$ with finitely many conjugacy classes of $g$ and $g^{-1}$. The number of conjugacy classes that one needs to generate $G(p)$ then depends on the projective distance of the element $g$ to the identity in the 1-norm $\norm{\cdot}_1=\tau(|\cdot|)$ as in \cite[Theorem 1.3]{DT-15}. 
\end{rem}

\subsection{Strong uncountable cofinality for properly infinite von Neumann algebras}\label{sec_inf}

In this section $M$ will denote a properly infinite von Neumann algebra. As explained at the beginning of Section \ref{sec_factors}, $\mathcal{P}$ then denotes the set of projections $p\in M$ such that both $p$ and $1-p$ is infinite. 
A crucial fact that we use throughout this section is that any two nonzero projections in $\mathcal{P}$ are conjugate by a unitary.

We first prove an analogue of Lemma \ref{lem_masterful} in the case of a properly infinite von Neumann algebra. We include the proof of Lemma \ref{lem_masterful_inf} for the sake of completeness, even though it is similar to the proof of Lemma \ref{lem_masterful}.

\begin{lem}\label{lem_masterful_inf}
 Let $M$ be a properly infinite von Neumann algebra. Let $p\in \mathcal{P}$ be a nonzero projection in $M$ and decompose it orthogonally into $p=p_1+p_2+p_3$, where $p_i\in\mathcal{P}$ for $i=1,2,3$. Then
 $$G(p)=(G(p_1+p_2)\cup G(p_2+p_3))^5.$$ 
\end{lem}

\begin{proof}
 Let $u\in G(p)$ be arbitrary and put $q_i\coloneqq up_iu^{-1}$ for $i=1,2,3$.  
 Let $e\coloneqq p_2+p_3$. Since $p_2$ and $p_3$ are infinite projections, there exists $u_1\in G(p_2+p_3)$ such that 
 $$u_1(e-e\wedge q_1^{\bot})u_1^{-1}=e-e\wedge u_1q_1^{\bot}u_1^{-1}\leq p_2$$
 is an infinite subprojection. 
 We have that $u_1q_1u_1^{-1}\bot p_3$. 
 Hence there exists $u_2\in G(p_1+p_2)$ such that $u_2u_1q_1u_1^{-1}u_2^{-1}=p_1$ and  $u_2u_1(q_2+q_3)u_1^{-1}u_2^{-1}\leq p_2+p_3$ is an infinite subprojection. So we can find $u_3\in G(p_2+p_3)$ with $u_3u_2u_1q_2u_1^{-1}u_2^{-1}u_3^{-1}=p_2$ and thus with $v\coloneqq u_3u_2u_1$ we have 
 $vq_iv^{-1}=p_i=u^{-1}q_iu \mbox{ for }i=1,2,3.$
 Thus $vup_i=p_i vu$, i.e. there exist $v_1\in G(p_1),\ v_2\in G(p_2),\ v_3\in G(p_3)$ such that $vu=v_1v_2v_3$, i.e. $u=u_1^{-1}u_2^{-1}u_3^{-1}v_1v_2v_3$, with $u_2,v_1v_2\in G(p_1+p_2)$ and $u_1,u_3, v_3\in G(p_2+p_3)$. Hence $u\in (G(p_1+p_2)\cup G(p_2+p_3))^5$, as claimed.  
\end{proof}

\begin{proof}[Proof of Theorem \ref{thm_main_inf}:]
 We use Theorem \ref{thm_Bergmann}. Recall that condition (A) is automatically satisfied for unitary groups of properly infinite von Neumann algebras. Thus we only have to verify conditions (B) and (C) in Theorem \ref{thm_Bergmann}.
 
 We verify condition (B). Observe that a symmetry in a von Neumann algebra is of the form $1-2p$, where $p$ is a projection. Thus \cite[Corollary]{Fill-66} implies that every element in $G(p)$, $p\in\mathcal{P}$ arbitrary, is a product of four symmetries of the form $1-2q_i$, where $q_i$ is an infinite subprojection of $p$ for $i=1,2,3,4$. As these projections are unitarily conjugate, we have that (B) in Theorem \ref{thm_Bergmann} holds with $n=4$.
   
 It remains to verify condition (C) in Theorem \ref{thm_Bergmann}. Let $p\in\mathcal{P}\setminus\{0,1\}$. Decompose $p=p_1+p_2$ orthogonally with $p_1,p_2\in\mathcal{P}\setminus\{0\}$. Put $p_3\coloneqq 1-p_1-p_2\in\mathcal{P}\setminus\{0\}$. Let $u\in G(p+p_3)$ be the unitary conjugating $p=p_1+p_2$ to $p_2+p_3$ and observe that $G(p_2+p_3)=(G(p)\cup\{u,u^{-1}\})^3$. Now we use Lemma \ref{lem_masterful_inf} to arrive at 
 $$G=(G(p)\cup G(p_2+p_3))^5=(G(p)\cup\{u,u^{-1}\})^8.$$
 Using Theorem \ref{thm_Bergmann} we conclude that $G$ has $96$-strong uncountable cofinality.
\end{proof}

\begin{rem}
For $G=\mathrm{U}(\mathcal{H})$, $\mathcal{H}$ an infinite-dimensional Hilbert space, condition (B) in the above proof also follows from \cite[Theorem 1]{HK-58}.

If $G$ denotes the projective unitary group of a type III factor, then condition (B) in the above proof also follows from bounded normal generation, see \cite[Theorem 1.3]{DT-16}, and one can generate $G(p)$ in finitely many steps using the conjugacy class of any element $g\in G(p)\setminus\{1\}$ and of its inverse. The number of conjugacy classes of $g$ that one needs to generate $G(p)$ then depends on the projective distance of $g$ to the identity in the operator norm.
\end{rem}

\section{Some questions on separable group topologies}\label{sec_open}
Recently Tsankov has shown in \cite[Corollary 2]{T-13} that $\mathrm{U}(\ell^2(\mathbb{N}))$ has a unique nontrivial separable group topology, namely the strong operator topology. On the other hand, the finite-dimensional unitary group $\mathrm{U}(n),\ n\in\mathbb{N},$ has at least two nontrivial separable group topologies: a unique Polish group topology by \cite{Ka-84} and a second separable group topology coming from a discontinuous embedding of $\mathrm{U}(n)$ into the Polish group $S_\infty$ of bijections $\mathbb{N}\to\mathbb{N}$ (with the topology of pointwise convergence) by \cite{Ka-00}, where we use that a subgroup of a Polish group is Polish if and only if it is closed (see \cite[Proposition 1.2.1]{BK-96}).  

A natural question extending the results mentioned in the paragraph above is to ask whether $\mathrm{U}(M)$ (or $\mathrm{PU}(M)$) has a unique nontrivial separable group topology whenever $M$ is an infinite-dimensional separable factor. But in this generality the answer is no, as the presumably known observation below shows. Recall that $\mathrm{PU}(M)=\mathrm{U}(M)/\mathcal{Z}(\mathrm{U}(M))$, where $\mathcal{Z}(\mathrm{U}(M))$ denotes the center of $\mathrm{U}(M)$.

\begin{obs}
	If $M$ is a separable von Neumann algebra which is not full, then $\mathrm{PU}(M)$ has at least two nontrivial separable group topologies.
\end{obs}
\begin{proof}
	Since $M$ is not full, by definition the group of inner automorphisms $\mathrm{Inn}(M)$ is not closed in the Polish group $\mathrm{Aut}(M)$ of automorphisms of $M$, equipped with topology of pointwise norm convergence on the predual $M_*$. As in the proof of \cite[Theorem XIV.3.8]{Tak} we have that $\mathrm{PU}(M)$ is isomorphic to $\mathrm{Inn}(M)$ under $\mathrm{Ad}$. Since $\mathrm{Inn}(M)$ is not closed in $\mathrm{Aut}(M)$, it cannot be Polish by \cite[Proposition 1.2.1]{BK-96}. Using the fact that a subspace of a separable metric space is separable, we obtain a separable group topology on $\mathrm{PU}(M)$ which is not Polish. 
\end{proof}

Examples of full factors are the group factors $L(\mathbb{F}_n)$ of the free group $\mathbb{F}_n$ in $n\geq 2$ generators by \cite[Chapter VI]{MvN-43}.
An example of a non-full II$_1$ factor is the hyperfinite II$_1$ factor $\mathcal{R}$ by \cite[Chapter VI]{MvN-43}, and hence $\mathrm{PU}(\mathcal{R})$ has at least two nontrivial separable group topologies - one being Polish (the strong operator topology induced by the 2-norm of the trace) and one only being separable. 

\begin{quest} 
	Does $\mathrm{U}(M)$ (or $\mathrm{PU}(M)$) have a unique nontrivial separable group topology whenever $M$ is a separable full von Neumann algebra? Does $\mathrm{U}(M)$ (or $\mathrm{PU}(M)$) have a coarsest separable group topology for any infinite-dimensional separable von Neumann algebra? 
\end{quest}

We remark that it is known from \cite[Theorem 1.4]{DT-15} that the projective unitary group of any separable II$_1$ factor has a unique Polish group topology. \\

Recall that a topological group has \textit{automatic continuity} if every homomorphism to any separable topological group is continuous. 
A positive answer to the first question would imply automatic continuity for the group in question, endowed with the strong operator topology. 
Similarly, a positive answer to the second question would imply automatic continuity of the given group with respect to its coarsest nontrivial separable topology. 
Note that it was proven in \cite[Theorem 1.4]{DT-15} that whenever $M$ is a separable finite factor (i.e. $M$ is finite-dimensional or of type II$_1$), every homomorphism from $\mathrm{PU}(M)$ to any separable SIN group is continuous.\\

Let us rephrase the question above in the case that $M$ is either a separable II$_1$ factor or a separable type III factor.
If $G$ is a topologically simple group with neutral element $1_G$, then the closure of $\{1_G\}$ is a closed normal subgroup of $G$, and hence any nontrivial group topology on $G$ must be Hausdorff.
The projective unitary group $\mathrm{PU}(M)$ of $M$ is simple by \cite{dlH-79} (this also follows from \cite[Theorem 1.2]{DT-15}), respectively \cite[Corollary 3.2]{DT-16}. Thus any nontrivial group topology $\mathcal{T}$ on $\mathrm{PU}(M)$ is Hausdorff. Assume that $\mathcal{T}$ is first countable. By the Birkhoff-Kakutani Theorem \cite[Theorem 9.1]{Kec} the topology $\mathcal{T}$ is metrizable with a compatible left-invariant metric $d$. Hence $d$ is bounded by Corollary \ref{cor_main}.  
Thus we can reformulate our question as follows:\\
Is a bounded left-invariant metric $d$ inducing a separable topology on $\mathrm{PU}(M)$ automatically complete for $\mathrm{PU}(M)$ whenever $M$ is a separable full II$_1$ factor or a separable type III factor? 

\section*{Acknowledgments}

The author wants to thank Vadim Alekseev, Manfred Droste and Stefaan Vaes for helpful discussions. The author thanks Hiroshi Ando, Yasumichi Matsuzawa and Andreas Thom for pointing out some mistakes. 
P.A.D. was supported by ERC CoG No. 614195 and by ERC CoG No. 681207.

\bibliographystyle{alpha}

\end{document}